\documentclass[11pt]{amsart}

\usepackage{amsmath, amssymb, amsthm, comment}

\usepackage{a4wide}

\theoremstyle{plain}
\newtheorem{theorem}{Theorem}[section]
\newtheorem{lemma}[theorem]{Lemma}
\newtheorem{propn}[theorem]{Proposition}
\newtheorem{corollary}[theorem]{Corollary}
\newtheorem{assumption}[theorem]{Assumption}

\theoremstyle{remark}
\newtheorem{remark}[theorem]{Remark}
\newtheorem*{acknow}{Acknowledgement}
\newtheorem{definition}[theorem]{Definition}
\newtheorem{example}[theorem]{Example}

\DeclareMathOperator\supp{supp}

\DeclareMathOperator\dom{dom}

\newcommand{\lt}{L}
\newcommand{\rt}{R}

\newcommand\sub{\subseteq}

\newcommand\conv{\star}

\newcommand{\Com}{\Delta}

\newcommand{\QG}{\G}

\DeclareMathOperator\clsp{\overline{\text{span}}}

\newcommand{\Z}{\mathbf{Z}}

\newcommand{\set}[2]{\{\,{\textstyle#1}:\,{\textstyle #2}\,\}}

\newcommand{\inv}{^{-1}}
\newcommand\conj{\overline}

\newcommand\lone{\mathrm{L}^1}
\newcommand\ltwo{\mathrm{L}^2}

\newcommand{\C}{\mathrm{C}}
\newcommand{\G}{\mathbb{G}}
\renewcommand{\H}{\mathbb{H}}

\newcommand{\fatnu}{\mathbf{\nu}}
\newcommand{\fatrho}{\mathbf{\rho}}

\newcommand{\ot}{\otimes}
\newcommand{\cop}{\Delta}
\newcommand{\id}{\mathrm{id}}
\newcommand{\M}{\mathrm{M}}

\newcommand{\Cou}{\epsilon}

\newcommand{\SU}{\mathrm{SU}}
\newcommand{\T}{\mathbf{T}}

\newcommand{\CG}{\C_0 (\G)}

\newcommand{\rom}{|\omega|_r}
\newcommand{\lom}{|\omega|_l}

\newcommand\vnot{\mathop{\overline\otimes}}
\newcommand{\wot}{\vnot}

\newcommand{\Hil}{\mathsf{H}}
\newcommand{\alg}{\mathsf{A}}
\newcommand{\clg}{\mathsf{C}}
\newcommand{\blg}{\mathsf{B}}

\newcommand{\Kil}{\mathsf{K}}

\newcommand{\K}{K}

\newcommand{\Alg}{\mathcal{A}}

\newcommand{\Npsi}{\mathfrak{N}_{\psi}}

\newcommand{\phitwo}{\phi^{(2)}}
\newcommand{\psitwo}{\psi^{(2)}}
\newcommand{\Mtwo}{M_2(\C_0(\G))}

\newcommand{\la}{\langle}
\newcommand{\ra}{\rangle}

\newcommand{\Lonenat}{\lone_{\sharp} (\G)}

\newenvironment{rlist}
{

\begin{enumerate}}
{\end{enumerate}}

\begin{document}

\title[Contractive idempotents on quantum groups]{Contractive idempotents on
       locally compact quantum groups}


\begin{abstract}
\noindent A general form of contractive idempotent functionals  on coamenable locally compact
quantum groups is obtained, generalising the result of Greenleaf on
contractive measures on locally compact groups. The image of a
convolution operator associated to a contractive idempotent is shown
to be a ternary ring of operators. As a consequence a one-to-one
correspondence between contractive idempotents and a certain
class of ternary rings of operators is established.
\end{abstract}

\author{Matthias Neufang}
\address{School of Mathematics \& Statistics, Carleton University, Ottawa, ON,
Canada K1S 5B6}
\address{ Universit\'e Lille 1 - Sciences et Technologies, UFR de Math\'ematiques, Laboratoire de Math\'e--matiques
Paul Painlev\'e - UMR CNRS 8524, 59655 Villeneuve d'Ascq C\'edex, France}
\email{mneufang@math.carleton.ca}

\author{Pekka Salmi}
\address{Department of Mathematical Sciences,
University of Oulu, PL 3000, FI-90014 Oulun yliopisto, Finland}
\email{pekka.salmi@iki.fi}

\author{Adam Skalski}
\address{Institute of Mathematics of the Polish Academy of Sciences,
ul.\'Sniadeckich 8, 00-956 Warszawa, Poland}
\email{a.skalski@impan.pl}

\author{Nico Spronk}
\address{Department of Pure Mathematics,
University of Waterloo, Waterloo, ON, N2L 3G1, Canada}
\email{n.spronk@uwaterloo.ca}

\maketitle

\section*{Introduction}

Idempotent probability measures on locally compact groups arise
naturally as limit distributions of random walks and have been studied
in this context at least since the 1940s (for the history of the
investigations of such idempotent measures we refer to the recent
survey \cite{Pekkasurvey}). They all turn out to arise as Haar
measures on compact subgroups of the original group. There are several
possible generalisations of the problem of characterising all
idempotent probability measures. On one hand one can ask about
\emph{all} idempotent measures. The problem turns out to be very hard
and so far has only been solved in the case of abelian groups in a
celebrated paper by P.\,Cohen (\cite{Cohen}).  In fact, a solution to this
problem for discrete groups would solve Kaplansky's idempotent
conjecture.  It is significantly
easier to analyse all \emph{contractive} (with respect to the total
variation norm) idempotent measures -- in 1965 F.P.\,Greenleaf
(\cite{greenleaf:homo}) showed that they arise as combinations of Haar measures
on compact subgroups and characters on those subgroups. 
Note also that the dual issue of providing the characterisation of idempotents in the Fourier-Stieltjes algebra of a locally compact group was resolved by B.\,Host in \cite{Host}. 
On the other
hand one can ask about the counterparts of probability measures on
locally compact \emph{quantum} groups, usually called idempotent
states. The study of the latter was initiated in \cite{FSJAlg} and
continued later for example in \cite{salmi-skalski:idem} (once again
the description of these developments and a full list of references
can be located in \cite{Pekkasurvey}). An important feature of
idempotent states is that they need not be \emph{Haar idempotents},
i.e.\ Haar states on compact quantum subgroups.

In this article we combine the two generalisations and investigate
\emph{contractive idempotents on locally compact quantum
  groups}. These are idempotent functionals on $\C_0(\G)$, the algebra
of continuous functions on a locally compact quantum group $\G$, and as
such  can be viewed as quantum counterparts of contractive idempotent
measures on locally compact groups. In particular they include
contractive idempotent functions in $B(G)$, the Fourier-Stieltjes
algebra of a locally compact group $G$. Similarly to the classical
case, contractive idempotents on $\G$ arise as limits of iterated
convolutions of a given contractive quantum measure. They also turn
out to be of great relevance for the study of the harmonic functions on quantum groups, i.e.\
fixed points of convolution operators associated with
contractive quantum measures (in the case of standard groups and their
duals such structures were studied for example in the book
\cite{Chu-Lau} and, more recently, in \cite{CL2};  for Poisson boundaries associated to \emph{states} on
locally compact quantum groups we refer to the recent preprint
\cite{KNR} and references therein).
For technical reasons we assume that the locally compact quantum
groups we consider are \emph{coamenable}. Under this assumption we
show that each contractive idempotent $\omega$ on a quantum group
$\QG$ is associated via its polar decomposition to two (possibly
identical) idempotent states. This leads to a general characterization
of the form of $\omega$ and can be further simplified in the case when
the idempotent states associated to $\omega$ are Haar idempotents.
In the latter case the resulting
structure precisely mirrors that obtained by Greenleaf for classical
groups: a contractive idempotent $\omega$ arises from a Haar state on
a compact quantum subgroup and a character (i.e.\ a group-like
element) on that subgroup. We show also that in this case the left
and right absolute values of $\omega$ coincide. These results contain
in particular the classical theorem of Greenleaf.

In \cite{salmi-skalski:idem} (see also \cite{S}) idempotent states on
$\G$ were shown to be closely related to \emph{right invariant
  expected} subalgebras of $\C_0(\G)$. It turns out that a similar
result holds true for contractive idempotents, but subalgebras have to
be replaced by \emph{ternary rings of operators (TROs)}. The appearance of
the TRO structure is related to the fact that the left
convolution operator associated to a contractive idempotent satisfies
the algebraic relations characterising the so-called TRO conditional
expectations. The TROs arising in this way are invariant under
right convolution operators; moreover their \emph{linking
  algebras} contained in two-by-two matrices over $\C_0(\G)$ admit
conditional expectations preserving a suitable amplification of the
left Haar weight. We show that  if $\G$ is \emph{unimodular}, then in
fact there is a bijective correspondence between such TROs in
$\C_0(\G)$ and contractive idempotents on $\G$. We would like to point
out that such TRO structures play a fundamental role in the study of
the extended Poisson boundaries for contractive quantum measures,
which we undertake in the forthcoming work \cite{NSSS}.

A detailed plan of the paper is as follows.  In Section 1 we
set up notation, introduce basic terminology of quantum group theory
and recall fundamental results on idempotent states. In Section
2 we characterise the general form of contractive idempotents and
analyse their connection to idempotent states.
In section 3 we study the properties of the convolution
operator associated with a contractive idempotent, showing in
particular that its image is a TRO; further in Section 4 we establish
a one-to-one correspondence between contractive idempotents and a
certain class of TROs. Finally Section 5 shows how our results allow
us to obtain quickly the known characterisations for $\G$ being either
a classical group or the dual of a classical group, and presents some
examples.

A word of warning is in place -- the quantum group notation used in
this paper differs from that in \cite{salmi-skalski:idem}; it is
however consistent with the notation more and more often adopted in
the modern literature of the subject.

\section{Preliminaries}

In this section we introduce our notation, prove a few technical results we will need later, and recall basic information on idempotent states on locally compact quantum groups.

\subsection{$C^*$-algebra notation}
\label{C*notations}

Let $\alg$ be a $C^*$-algebra. The state space of $\alg$ will be
denoted $S(\alg)$,  the multiplier algebra of $\alg$ by $M(\alg)$. We
will often use the strict topology and strict extensions of maps to
multiplier algebras without any comment -- the detailed arguments can
be found in \cite{Lance} or in Section 1 of
\cite{salmi-skalski:idem}. The symbol $\ot$ will denote the spatial
tensor product of $C^*$-algebras and $\wot$ will be reserved for the
von Neumann algebraic tensor product. If $\Hil$ is a Hilbert space
then $\K_{\Hil}$ denotes the algebra of compact operators on $\Hil$. If $\Kil$ is an additional Hilbert space and $X\subset B(\Hil, \Kil)$, $Y\subset B(\Kil, \Hil)$  we write $X^*=\{x^*:x\in X\}$ and $\langle X Y \rangle= \overline{\textup{span}} \{xy: x\in X, y\in Y\}$.

If $\omega \in S(\alg)$ then
we write
$N_{\omega}=\{a\in \alg: \omega(a^*a)=0\}$.

If  $\omega \in \alg^*$, its
unique normal extension to a functional on the von Neumann algebra
$\alg^{**}$ will be sometimes denoted by the same symbol.

We have natural left and right actions of $\alg$ on $\alg^*$. They can be extended in a natural sense to actions of $\alg^{**}$:
 for $\omega\in\alg^*$, $x,y\in\alg^{**}$,
write
\[
\omega.x(y) = \omega(xy)\qquad\text{and}\qquad
x.\omega(y) = \omega(yx).
\]

By \cite{takesaki:vol1}, given a functional $\omega \in \alg^*$ there
is a (unique) polar decomposition
\[
\omega = u.|\omega|_r
\]
where $|\omega|_r$ is a positive functional in $\alg^*$, called the \emph{right absolute value} of $\omega$,
and $u$ is a partial isometry in $\alg^{**}$ such that
$u^*u$ is the support of $|\omega|_r$
(i.e.\ $u^*u$ is the minimal projection
satisfying $|\omega|_r(u^*u x) = |\omega|_r(xu^*u) = |\omega|_r(x)$).
It follows that
\[
|\omega|_r = u^*.\omega.
\]
Putting $|\omega|_l = u.|\omega|_r.u^*$
we obtain another positive functional $|\omega|_l$ (the \emph{left absolute value}  of $\omega$) such that
\[ \omega = |\omega|_l.u,\qquad
|\omega|_l = \omega.u^*.
\]
(Now $uu^*$ is the support of $|\omega|_l$.)

\subsection{Locally compact quantum groups -- general background}

We follow the $C^*$-algebraic approach to locally compact quantum groups
due to Kustermans and Vaes \cite{KV}.

A \emph{comultiplication} on a $C^*$-algebra $\alg$ is a nondegenerate
$^*$-homomorphism $\Com: \alg\to M(\alg\otimes\alg)$ that is
coassociative:
\[
(\id_\alg\otimes \Com)\Com = (\Com\otimes\id_\alg)\Com.
\]
A $C^*$-algebra $\alg$ with a comultiplication $\Com$
is \emph{the algebra of continuous functions vanishing at infinity on
 a locally compact quantum group} if $(\alg, \cop)$ satisfies
the quantum cancellation laws
\begin{equation}
\clsp\Com(\alg)(\alg\otimes 1_\alg) = \alg\ot \alg,\quad
\clsp\Com(\alg)(1_\alg\otimes\alg)  = \alg\ot \alg
\label{cancellation} \end{equation}
and admits a left Haar weight $\phi$ and a right Haar weight $\psi$.
That is, $\phi$ and $\psi$ are faithful KMS-weights on $\alg$
such that $\phi$ is \emph{left invariant}
\[
\phi((\omega\ot\id)\Com(a)) = \omega(1_\alg)\phi(a)
\qquad (\omega\in \alg^*_+, a\in \alg_+, \phi(a)< \infty )
\]
and $\psi$ is \emph{right invariant}
\[
\psi((\id\ot\omega)\Com(a)) = \omega(1_\alg)\psi(a)
\qquad (\omega\in \alg^*_+, a\in \alg_+, \psi(a)< \infty)
\]
(for the meaning of the KMS property we refer the reader to \cite{KV}).
In such a case we use the notation $\alg=\C_0(\QG)$ and call $\QG$ a
\emph{locally compact quantum group}.

If the $C^*$-algebra $\C_0(\QG)$ is unital, we say that the quantum group
$\QG$ is \emph{compact} and write $\C(\G)$ instead of $\C_0(\G)$. In this case there is a unique \emph{Haar state}
of $\QG$ that is both left and right invariant. A unitary $u\in \C(\G)$ is said to be \emph{group-like} if $\Com(u) = u \ot u$.

A \emph{counit} is a linear functional $\epsilon$ on $\C_0(\QG)$ such that
\[
(\epsilon\ot\id_{\C_0(\QG)})\Com = (\id_{\C_0(\QG)}\ot\epsilon)\Com = \id_{\C_0(\QG)}.
\]
A quantum group that admits a bounded counit is said to be
\emph{coamenable} \cite{bedos-tuset}.

\begin{assumption}
From now on we assume that all the locally compact quantum groups studied in this paper are coamenable.
\end{assumption}

Note that due to the assumption of coamenability,
the $C^*$-algebra $\C_0(\G)$ coincides with the universal $C^*$-algebra
$\C_0^u(\G)$  associated to $\G$ (as defined by Kustermans
\cite{kus:univ}). The multiplier algebra of $\C_0(\G)$ will be sometimes denoted $\C_b(\G)$.

The GNS representation space for the left Haar weight will be denoted
by $\ltwo(\QG)$. We may in fact assume that $\C_0(\QG)$ is a
nondegenerate subalgebra of $B(\ltwo(\QG))$.
The Banach space dual of $\C_0(\G)$ will be denoted $\M(\G)$ and
called the \emph{measure algebra of $\QG$}.
It is a Banach algebra with the product defined by
\[ \omega \conv \mu := (\omega \ot \mu) \circ \Com, \;\;\; \omega, \mu
\in \M(\G),\]
and since $\G$ is coamenable, $\M(\G)$ is unital.

Given $\omega \in \M(\G)$ the associated \emph{left convolution} operator $L_{\omega}:\C_0(\G) \to \C_0(\QG)$ is defined by the formula
\[ \lt_{\omega} (a) = (\omega \ot \id_{\C_0(\G)}) (\Delta(a)), \;\;\; a\in \C_0(\G).\]
(the right convolution operators $\rt_{\omega}$ are defined in an analogous way). They are automatically strict and hence extend to maps on $\C_b(\G)$.
The following lemma from \cite{salmi-skalski:idem} will allow us later to identify the (left) convolution operators.

\begin{lemma}[Lemma 1.6 of \cite{salmi-skalski:idem}] \label{invar}
Let $T:\C_0(\G) \to \C_0(\G)$ be a completely bounded map such that $T\ot\id_{\C_0(\G)}: \C_0(\G) \ot \C_0(\G) \to \C_0(\G) \ot \C_0(\G)$ is
strictly continuous on bounded subsets. Suppose that $Z\sub \M(\G)$ is weak$^*$-dense and invariant under the right action of some dense subalgebra $\Alg$ of $\C_0(\G)$. Then the following conditions are equivalent:
\begin{rlist}
\item $T=\lt_{\mu}$, for some $\mu \in \M(\G)$;
\item \label{commutT}
      $(T \ot \id_{\alg} ) \Com = \Com  T$;
\item $T \rt_{\nu} = \rt_{\nu} T$ for all $\nu \in Z$.
\end{rlist}
If the above conditions hold, $T= \rt_{\Cou \circ T}$ is a linear combination of completely positive nondegenerate
maps and the equality in (iii) is valid for all $\nu \in \M(\G)$.
\end{lemma}

\begin{definition}
A $C^*$-subalgebra $\clg \sub \C_0(\G)$ is said to be \emph{right
invariant} if $R_{\mu}(\clg) \sub \clg$ for
all $\mu \in \M(\G)$.
\end{definition}

Define $\lone(\QG)$ to be the collection of those functionals in
$\M(\G)$ which admit normal extensions to $B(\ltwo(\G))$.
The \emph{antipode} $S$ of $\G$ is a (possibly unbounded) operator
on $\C_0(\G)$; the set
\[
\set{(\sigma\ot\id)(V)}{\sigma\in B(\ltwo(\G))_*},\]
where $V\in B(\ltwo(\G)\ot_2\ltwo(\G))$
is the \emph{right multiplicative unitary} of $\G$,
gives a core for $S$ and
\[
S\bigl((\sigma\ot\id)(V)\bigr) = (\sigma\ot\id)(V^*),\;\;\; \sigma\in B(\ltwo(\G))_*.
\]
For $\omega\in\M(\G)$, define $\conj{\omega}\in\M(\G)$ by
$\conj{\omega}(a) = \conj{\omega(a^*)}$.
When $\conj{\omega}\circ S$ (defined on $\dom(S)$)
has a \emph{bounded} extension to all of $\C_0(\G)$,
we write $\omega^\sharp$ for this functional.
Then $\Lonenat$ denotes those $\omega$ in $\lone(\G)$
for which $\overline{\omega} \circ S$ is bounded.

\subsection{Idempotent states on locally compact quantum groups}

\begin{definition}
An \emph{idempotent functional} in $\M(\QG)$ is a nonzero functional $\omega \in \M(\G)$ such that $\omega \conv \omega =\omega$.
\end{definition}

Idempotent states on $\M(\G)$ were studied in \cite{salmi-skalski:idem} (see also \cite{FST} for the compact case). Here we quote the main results and some auxiliary lemmas we will use later. First we need some further definitions.

\begin{definition}
Assume that $\G$ is a coamenable locally compact quantum group and
$\H$ a compact quantum group. Then $\H$ is a \emph{closed quantum
  subgroup} of $\G$ if there exists a surjective $^*$-homomorphism
$\pi_{\H}:\C_0(\G) \to \C(\H)$ intertwining the respective coproducts
(sometimes we shorten this to write that $(\H, \pi)$ is a compact
quantum subgroup of $\G$).
\end{definition}

Note that there are various possible definitions of locally compact quantum subgroups; all of them however coincide when the quantum subgroup in question is compact (see \cite{DKSS} for a complete discussion).

\begin{definition}
An idempotent state $\omega \in \M(\G)$ is a \emph{Haar idempotent} if there
exists a compact quantum subgroup $\H$ of $\G$ such that $\omega =
h_{\H} \circ \pi_{\H}$, where  $h_{\H}$ denotes the Haar state on $\H$.
\end{definition}

\begin{theorem} [Theorem 3.7 of \cite{salmi-skalski:idem}] \label{thm:Haar-equiv}
Let $\omega$ be an idempotent state in $\M(\G)$.
Then the following are equivalent:
\begin{rlist}
\item $\omega$ is a Haar idempotent;
\item $N_{\omega}$ is an ideal
      \textup{(}equivalently a $^*$-subspace\textup{)}.
\end{rlist}
\end{theorem}

\section{General form of contractive idempotent functionals}

In this section we describe the general form of contractive
idempotents in $\M(\G)$. Note that a contractive idempotent must
necessarily have norm $1$. We begin by showing how one can construct
them using idempotent states and elements in $\C_b(\G)$
satisfying a version of the group-like property.

\begin{propn} \label{prop:construct}
Suppose that $\sigma$ is an idempotent state in $\M(\G)$,
and $u\in \C_b(\G)$ satisfies
\[
\cop(u)-u\ot u\in N_{\sigma\ot\sigma} \sub M(\C_0(\G)\ot \C_0(\G)).
\]
Then either $\sigma(u^*u) = 1$ or $\sigma(u^*u) = 0$.
In the former case $u.\sigma.u^*$ is an idempotent state and
$u.\sigma$ and  $\sigma.u^*$ are contractive idempotents.
\end{propn}

\begin{proof}
The Cauchy--Schwarz inequality implies that
for every $X$ in $M(\C_0(\G)\ot \C_0(\G))$
\[
(\sigma\ot\sigma)((u^*\ot u^*)X)
= (\sigma\ot\sigma)(\cop(u^*)X)
\quad\text{and}\quad
(\sigma\ot\sigma)(X(u\ot u))
= (\sigma\ot\sigma)(X\cop(u)).
\]
In particular, using the assumption that $\sigma$ is idempotent,
we see that $\sigma(u^*u)^2 = \sigma(u^*u)$,
so $\sigma(u^*u)$ is either $1$ or $0$.
(Note that if $\sigma(u^*u) = 0$, then
$u.\sigma = \sigma.u^* = u.\sigma.u^* = 0$.)

Moreover,
\[
(u.\sigma\conv u.\sigma)(a)
=(\sigma\ot\sigma)\bigl(\cop(a)(u\ot u)\bigr)
= (\sigma\ot\sigma)(\cop(au))
= u.\sigma(a)
\]
for every $a$ in $\C_0(\G)$. Therefore $u.\sigma$ is a contractive
idempotent when $\sigma(u^*u)\neq 0$. The other cases are similar.
\end{proof}

We will now show that in fact all contractive idempotents in $\M(\G)$
have the form described above. We begin with the following lemma, which
is a generalisation to quantum groups of Theorem~2.1.2 in
\cite{greenleaf:homo}. Recall the polar decomposition of functionals
described in Subsection \ref{C*notations}.

\begin{lemma} \label{lemma}
Let $\omega_1$, $\omega_2 \in \M(\G)$ be
such that $\|\omega_1\conv\omega_2\| = \|\omega_1\| \|\omega_2\|$.
Then $|\omega_1\conv\omega_2|_r = |\omega_1|_r \conv |\omega_2|_r$.
\end{lemma}

\begin{proof}
Let $u_1$, $u_2$ and $u_{12}$ denote the partial isometries in $\C_0(\G)^{**}$
associated with the polar decompositions
of $\omega_1$, $\omega_2$ and $\omega_1\conv\omega_2$,
respectively. Then, for every $x$ in $\C_0(\G)^{**}$,
\begin{align*}
|\omega_1\conv\omega_2|_r(x) &= \omega_1\conv\omega_2(xu_{12}^*)
= (\omega_1\ot\omega_2)(\cop(x u_{12}^*)) \\
&= (|\omega_1|_r\ot|\omega_2|_r)(\cop(x) \cop(u_{12}^*)(u_1\ot u_2)).
\end{align*}
By the Cauchy--Schwarz inequality,
\begin{align*}
\bigl||\omega_1\conv\omega_2|_r(x)\bigr|^2
&\le
(|\omega_1|_r\ot|\omega_2|_r)(\cop(xx^*))
(|\omega_1|_r\ot|\omega_2|_r)((u_1^*\ot u_2^*)\cop(u_{12}u_{12}^*)(u_1\ot u_2))\\
&\le
(|\omega_1|_r\conv|\omega_2|_r)(xx^*)
(|\omega_1|_r\ot|\omega_2|_r)(u_1^*u_1\ot u_2^*u_2)\\
&= \|\omega_1\|\, \|\omega_2\| (|\omega_1|_r\conv|\omega_2|_r)(xx^*)
= \| \omega_1 \conv \omega_2\| (|\omega_1|_r\conv|\omega_2|_r)(xx^*)
\end{align*}
(the last identity is due to the hypothesis).
It then follows from the uniqueness of
the absolute value
$|\omega_1\conv\omega_2|_r$ (Proposition 4.6 of \cite{takesaki:vol1})
that
\[
|\omega_1\conv\omega_2|_r = |\omega_1|_r \conv|\omega_2|_r .
\]
\end{proof}

The next theorem is the main result of this section.
Comparing the theorem with Proposition~\ref{prop:construct}
shows that we have a characterisation of contractive idempotents.
We shall use the notation set up in the theorem for the
rest of the paper; that is, given a contractive idempotent
$\omega$, we let $v$ denote an element satisfying the following theorem.

\begin{theorem} \label{thm:1st-fact}
Let $\omega$ be a contractive idempotent in $\M(\G)$.
Then the absolute values $|\omega|_r$ and $|\omega|_l$ are idempotent
states and there exists $v\in \C_0(\G)$ such that
$\omega =|\omega|_l.v = v.|\omega|_r $ and
\[
\cop(v)-v\ot v\in N_{\rom\ot \rom} \cap (N_{\lom\ot \lom})^*.
\]
\end{theorem}

\begin{proof}
Note that $\|\omega\| = 1$. By Lemma~\ref{lemma}
\[
|\omega|_r\conv|\omega|_r = |\omega|_r
\]
so $|\omega|_r$ is an idempotent state.

Let $u\in \C_0(\G)^{**}$ be the partial isometry associated with the polar decomposition
of $\omega$.
Fix $a\in\C_0(\G)$ such that $\|\lt_\omega(a)\| = 1$
(such $a$ exists because otherwise $\omega=0$).
Choose $\nu\in\C_0(\G)^*$ such that $\|\nu\|\le 1$
and $\nu(\lt_\omega(a))= 1$.
Put $v = \rt_\nu(\lt_\omega(a))^*$, so $\|v\|\leq 1$.
Since $\omega$ is an idempotent,
\[
1 = \nu\bigl(\lt_\omega(a))\bigr) = \omega\bigl(\rt_\nu(\lt_\omega(a))\bigr)
= \omega(v^*).
\]
Hence $\|v\| = 1$. Moreover,
\begin{align*}
|\omega|_r \bigl((v-u)^*(v-u)\bigr) &=
|\omega|_r (v^*v) - |\omega|_r (v^*u) - |\omega|_r (u^*v) +  |\omega|_r (u^*u)\\
&=|\omega|_r (v^*v) - \omega(v^*) - \conj{\omega (v^*)} +  |\omega|_r (u^*u)
\le   \| v\|^2 - 1 - 1 + 1 = 0.
\end{align*}
It follows that $\omega = v.|\omega|_r$ and $|\omega|_r (v^*v) = 1$.

Now
\begin{align*}
&(|\omega|_r\ot|\omega|_r)
   \bigl((\cop(v)-v\ot v)^*(\cop(v)-v\ot v)\bigr) \\
&\qquad= (|\omega|_r\ot|\omega|_r)\bigl(\cop(v^*v)\bigr)
-(|\omega|_r\ot|\omega|_r)\bigl(\cop(v^*)(v\ot v)\bigr)\\
&\qquad\quad-(|\omega|_r\ot|\omega|_r)\bigl((v^*\ot v^*)\cop(v)\bigr)
+(|\omega|_r\ot|\omega|_r)(v^* v\ot v^* v)\\
&\qquad=|\omega|_r(v^*v)
-(\omega\ot\omega)(\cop(v^*))
-\conj{(\omega\ot\omega)(\cop(v^*))}
+|\omega|_r(v^* v)^2\\
&\qquad= 2 - \omega(v^*)-\conj{\omega(v^*)} = 0.
\end{align*}
Hence $\cop(v)-v\ot v\in N_{\rom\ot\rom}$.

Finally, $v.|\omega|_r.v^* = u.|\omega|_r.u^* = |\omega|_l$,
so also  $|\omega|_l$ is an idempotent state
by Proposition~\ref{prop:construct}.
Using similar arguments as those applied to $|\omega|_r$,
we obtain that $\omega = |\omega|_l.v$
and, moreover, $\cop(v)-v\ot v\in (N_{\lom\ot \lom})^*$.
\end{proof}

\begin{remark}
We know from the classical case due to Greenleaf \cite{greenleaf:homo}
that every contractive idempotent measure on a locally compact group $G$
is of the form $d\mu(s) = \chi(s)\,dm_H(s)$ where
$m_H$ is the Haar measure of a compact subgroup
$H$ of $G$ and $\chi$ is a continuous character of $H$.
This characterisation follows also from Theorem~\ref{thm:1st-fact}
(see Proposition~\ref{groupcontidemp} below);
$|\omega|_r = m_H$ and
the character $\chi$ is just the restriction of $v\in\C_0(G)$
to $H$.

Now consider the dual case when
$\C_0(\G)=\C^*(\Gamma)$ for a locally compact amenable group $\Gamma$
(sometimes one writes then $\G=\hat{\Gamma}$).
Ilie and Spronk \cite[Theorem 2.1]{IS1}
showed that a contractive idempotent
in the Fourier--Stieltjes algebra $B(\Gamma) \cong \C^*(\Gamma)^*$
is a characteristic function $1_C$ of an open coset $C$ of $\Gamma$
(even when $\Gamma$ is not amenable).
The most natural factorisation of $1_C$ is certainly
$1_C = 1_H(\cdot s) = \lambda(s).1_H$ where
$H$ is an open subgroup of $\Gamma$ with $C = H s\inv$.
However, $\lambda(s)$ is in the multiplier algebra $M(\C^*(\Gamma))$
and not in general in $\C^*(\Gamma)$.
A factorisation that does fit into the description of
Theorem~\ref{thm:1st-fact} is given by
$1_C =  \lambda(f).1_H$ where $f\in\lone(\Gamma)$ is
such that $\supp f\sub s H$ and $\int_G f(t)\,dt = 1$.
We note that also the description of contractive idempotents
in $B(\Gamma)$ (for amenable $\Gamma$) as characteristic functions
of open cosets follows from our results; see Proposition~\ref{dual}.
\end{remark}

\begin{theorem} \label{thm:cts-fact}
Let $\omega$ be a contractive idempotent in $\M(\G)$.
Suppose that $|\omega|_r$ is a Haar idempotent
and let $(\H,\pi)$ be the associated compact quantum subgroup
\textup{(}so that $\pi: \C_0(\G)\to\C(\H)$ is the morphism\textup{)}.
Then there exists a group-like unitary $u \in \C(\H)$
such that
\[ \omega = h_{\H} (\pi(\cdot) u).\]
Moreover,  $|\omega|_r = |\omega|_l$.
\end{theorem}

\begin{proof}
Let $v$ be as in Theorem \ref{thm:1st-fact} and write $u = \pi(v)$.
Let $h_\H$ be the Haar state of $\H$ so that
$|\omega|_r = h_\H\circ\pi$.
Since  $N_{\rom\ot\rom} = \ker \pi\ot\pi$
(because $\rom\ot\rom = (h_\H\ot h_\H)\circ(\pi\ot\pi)$ and
$h_\H\ot h_\H$ is faithful),
we have
\[
\cop_\H(u) = u\ot u
\]
by Theorem~\ref{thm:1st-fact}. Then
\[
h_\H(uu^* )1 =
(h_\H\ot\id)\cop_\H(uu^*) =
h_\H(uu^*)uu^*.
\]
Since $h_\H$ is faithful and $u\ne 0$
(because $\omega \ne 0$),  it follows that $uu^* = 1$.
Similarly $u^*u = 1$,
so $u$ is a group-like unitary in $\C(\H)$.

It remains to prove that $\lom =\rom$.
Consider the state $h':= u.h_{\H}.u^*\in \C(\H)^*$.
It is easy to check that $h'$ is an idempotent state
(see Proposition~\ref{prop:construct}).
As it is faithful on $\C(\H)$ it must be in fact equal to $h_{\H}$ (see Lemma 2.1 of \cite{woronowicz98}).
Let then $a\in \C_0(\G)$ and compute:
\begin{align*}
\lom(a) &= \rom(v^*av) = h_{\H} (\pi (v^*av))
         = h_{\H} (u^* \pi(a) u) \\
        &= h'(\pi(a)) =  h_{\H}(\pi(a)) =  \rom(a).
\end{align*}
\end{proof}

\begin{remark}
Naturally the above theorem has also a `left' version -- one can begin
by assuming that $\lom$ is a Haar idempotent (the latter is clearly
true if and only if $\rom$ is a Haar idempotent).
\end{remark}

\section{TROs associated with contractive idempotents} \label{TROSection}

In this section we show that to each contractive idempotent in
$\M(\G)$ one can associate in a natural way a ternary ring of
operators in $\C_0(\QG)$. We begin by introducing some definitions.

A \emph{ternary ring of operators} (\emph{TRO}) acting between Hilbert spaces
$\Hil$ and $\Kil$ is a closed subspace $T$ of $B(\Hil;\Kil)$ that
is closed under the ternary product:
\[
ab^*c \in T \qquad\text{whenever }a,b,c\in T.
\]
If $X\sub T$ is a sub-TRO (i.e.\ a closed subspace
that is closed under the ternary product), then a
contractive linear map $P$ from $T$ onto $X$ is called a
\emph{TRO conditional expectation} if
\[
P(ax^*y) = P(a)x^* y,\qquad
P(xa^*y) = xP(a)^* y,\qquad
P(xy^*a) = x y^*P(a),\qquad
\]
for every $a\in T$ and $x,y\in X$. TRO conditional expectations are automatically completely contractive (\cite{EOR}).

We shall show that the closed right invariant subspace
$\clg_\omega = \lt_\omega(\C_0(\G))$
associated to a contractive idempotent $\omega$ is
a sub-TRO of $\C_0(\G)$ and that $\lt_\omega: \C_0(\G) \to \clg_\omega$
is a TRO conditional expectation.

\begin{lemma} \label{lemma:lt_omega}
Let $\omega$ be a contractive idempotent in $\M(\G)$ and let  $a,b\in\C_0(\G)$. Then  % ^{**}
\begin{enumerate}
\item $\lt_\omega(\lt_\omega(a)b) = \lt_\omega(a)\lt_{|\omega|_\ell}(b)$;
\item $\lt_{|\omega|_\ell}(\lt_\omega(a)^*b) = \lt_\omega(a)^*\lt_{\omega}(b)$;
\item $\lt_\omega(a\lt_\omega(b)) = \lt_{|\omega|_r}(a)\lt_\omega(b)$;
\item $\lt_{|\omega|_r}(a\lt_\omega(b)^*) = \lt_\omega(a)\lt_\omega(b)^*$.
\end{enumerate}
\end{lemma}

\begin{proof}
Let $a,b \in \C_0(\G)$. Denote by $\Com^{(2)}:\C_0(\G) \to \M(\C_0(\G) \ot \C_0(\G) \ot \C_0(\G))$ the map given by $(\Com \ot \id)\circ\Com$ (equivalently  $(\id \ot \Com)\circ\Com$).  Then
\begin{align*}
\lt_\omega(\lt_\omega(a)b)
&= (\omega\ot\id)
  \bigl((\omega\ot\id\ot\id)(\cop^{(2)}(a)) \cop(b)\bigr)\\
&= (|\omega|_l\ot|\omega|_l\ot\id)
  \bigl((v\ot v\ot 1)(\cop^{(2)}(a))(1\ot \cop(b))\bigr)\\
&=  (|\omega|_l\ot|\omega|_l\ot\id)
  \bigl((\cop(v)\ot 1)(\cop^{(2)}(a))(1\ot \cop(b))\bigr)
\end{align*}
because $\cop(v)-v\ot v \in (N_{|\omega|_l\ot|\omega|_l})^*$.
Continuing the calculation we have
\begin{align*}
\lt_\omega\bigl(\lt_\omega(a)b\bigr)
&= (|\omega|_l\ot|\omega|_l\ot\id)
  \bigl((\cop\ot\id)((v\ot 1)\cop(a))(1\ot \cop(b))\bigr)\\
&= (|\omega|_l\ot\id)
  \bigl( (\lt_{|\omega|_l}\ot\id)((v\ot 1)\cop(a))\cop(b)\bigr)\\
&= (|\omega|_l\ot\id) \bigl((\lt_{|\omega|_l}\ot\id)((v\ot 1)\cop(a))\bigr)
  (|\omega|_l\ot\id)(\cop(b))
\end{align*}
because the multiplicative domain of the
idempotent state $|\omega|_l$ contains
$L_{|\omega|_l}(\C_0(\G))$ by Lemma 2.5 of \cite{salmi-skalski:idem}.
Rewriting the last expression, we obtain the first statement:
\[
\lt_\omega(\lt_\omega(a)b) =  \lt_\omega(a)\lt_{|\omega|_\ell}(b).
\]

Next we confirm the second statement,
using similar kind of manipulations as above:
\begin{align*}
&\lt_{|\omega|_l}(\lt_\omega(a)^*b)
= (|\omega|_l\ot\id)\bigl(\cop(b^*)\cop(\lt_\omega(a))\bigr)^*\\
&\qquad
= (\omega\ot\id)\bigl((v^*\ot 1)\cop(b^*)\cop(\lt_\omega(a))\bigr)^*\\
&\qquad= (|\omega|_r\ot\id)\bigl((v^*\ot 1)\cop(b^*)
                                 \cop(\lt_\omega(a))(v\ot 1)\bigr)^*\\
&\qquad= (|\omega|_r\ot|\omega|_r\ot\id)
\bigl((1\ot v^*\ot 1)(1\ot \cop(b^*))\cop^{(2)}(a)(v\ot v\ot 1)\bigr)^*\\
&\qquad= (|\omega|_r\ot\id)\bigl((v^*\ot 1)\cop(b^*)
         (\lt_{|\omega|_r}\ot\id)(\cop(a)(v\ot 1))\bigr)^*\\
&\qquad=
(|\omega|_r\ot\id)\bigl(\cop(a)(v\ot 1)\bigr)^*
(|\omega|_r\ot\id)\bigl((v^*\ot 1)\cop(b^*)\bigr)^*
         \\
&\qquad=\lt_\omega(a)^*\lt_\omega(b).
\end{align*}

The last two statements are right-hand analogues of the
first two and follow if the first two statements are applied
to the contractive idempotent $\conj{\omega}$, which is
defined by $\conj{\omega}(a) = \omega(a^*)^*$.
(Note that $|\conj\omega|_r =|\omega|_l$.)
\end{proof}

Before we formulate the main theorem of this section we present one
consequence of the above formulas: contractive idempotents are
automatically invariant under the adjoint map~$^{\sharp}$.

\begin{corollary}
If $\omega \in \M(\G)$ is a contractive idempotent
then $\omega^\sharp = \omega$.
\end{corollary}

\begin{proof}
The proof is similar to the proof of Proposition 2.6 of
\cite{salmi-skalski:idem} concerning idempotent states.
However, since $\lt_\omega$ is no longer a conditional expectation,
we sketch a proof.

Take the right multiplicative unitary $V$ of $\G$ and
consider $p = (\id\ot \omega)V$.
Then $(\id\ot\lt_\omega)(V) = (p\ot 1)V$.
By Lemma~\ref{lemma:lt_omega} (4), used in the third identity,
\begin{align*}
pp^*\ot 1 &= (p\ot 1)V V^* (p^*\ot 1)
= (\id\ot\lt_\omega)(V)(\id\ot\lt_\omega)(V)^*
= (\id\ot\lt_{|\omega|_r})\bigl(V(\id\ot\lt_\omega)(V)^*\bigr)\\
&= (\id\ot\lt_{|\omega|_r})\bigl(VV^*(p^*\ot1)\bigr)
= p^*\ot 1.
\end{align*}
It follows that $p = p^*$ and so
\[
\omega\bigl((\sigma\ot \id)(V)\bigr)
= \sigma(p) = \sigma(p^*) =
\conj{\omega}\bigl((\sigma\ot \id)(V^*)\bigr)
= \omega^\sharp\bigl((\sigma\ot \id)(V)\bigr)
\]
for every $\sigma\in B(\ltwo(\G))_*$.
\end{proof}

\begin{theorem} \label{thm:TRO expectation}
Let $\omega$ be a contractive idempotent in $\M(\G)$.
Then $\clg_\omega = \lt_\omega(\C_0(\G))$ is a TRO
and $\lt_\omega$ is a TRO conditional expectation.
\end{theorem}

\begin{proof}
Both statements follow from the identities
\begin{align*}
\lt_\omega(a)\lt_\omega(b)^*\lt_\omega(c)
&=\lt_\omega\bigl(\lt_\omega(a)\lt_\omega(b)^*c\bigr)\\
&=\lt_\omega\bigl(\lt_\omega(a)b^*\lt_\omega(c)\bigr)\\
&=\lt_\omega\bigl(a\lt_\omega(b)^*\lt_\omega(c)\bigr)
\end{align*}
where $a, b, c\in \C_0(\G)$.
We now check the first identity, the other two being
of similar flavour (in fact,
the first identity already implies the other two by the general theory
of completely contractive projections; see Corollary 3 of
\cite{youngson:cc-proj}). By Lemma~\ref{lemma:lt_omega} (1) and (2) we have
\[
\lt_\omega\bigl(\lt_\omega(a)\lt_\omega(b)^*c\bigr)
= \lt_\omega(a)\lt_{|\omega|_l}(\lt_\omega(b)^*c)
= \lt_\omega(a)\lt_\omega(b)^*\lt_\omega(c).
\]
\end{proof}

Every TRO $T\subset B(\Hil;\Kil)$ is associated with three $C^*$-algebras:
the \emph{left linking algebra} $\langle T T^*\rangle\subset B(\Kil)$,
the \emph{right linking algebra} $\langle T^* T\rangle \subset B(\Hil)$
and the \emph{linking algebra}
\[
\alg_T = \begin{bmatrix}
 \langle T T^*\rangle & T \\
 T^* & \langle T^* T\rangle
        \end{bmatrix} \subset B(\Kil \oplus \Hil).
\]
We say that a sub-TRO $X\sub T$ is \emph{nondegenerate} if
\[
\langle X T^* T\rangle = T\qquad\text{and}\qquad
\langle T T^* X\rangle = T.
\]
Equivalently, the linking $C^*$-algebra $\alg_X$ is
nondegenerate in $\alg_T$.
We note that for every contractive idempotent $\omega$
the sub-TRO $\lt_\omega(\C_0(\G))$ is
nondegenerate in $\C_0(\G)$. Indeed,
\[
\lt_\omega(\C_0(\G)) \C_0(\G) =
(\omega\ot\id)\bigl(\cop(\C_0(\G)) (1\ot\C_0(\G))\bigr)
\]
and thus it follows from the quantum cancellation laws \eqref{cancellation} that
$\lt_\omega(\C_0(\G)) \C_0(\G)$ is linearly dense in $\C_0(\G)$.

Theorem 2.1 of \cite{TRO-paper} shows that a TRO conditional
expectation $P: T\to X$ onto a nondegenerate sub-TRO $X\sub T$
extends (in a unique way) to a $C^*$-algebra conditional expectation
$E: \alg_T \to \alg_X$ between the linking algebras.
Explicitly,
\[
E = \begin{bmatrix}  P P^\dag & P \\
                    P^\dag    & P^\dag P
    \end{bmatrix}
\]
where
\[
P^\dag(a) = P(a^*)^*\qquad
P P^\dag(ax^*) = P(a) x^* \qquad
P^\dag P(x^*a) = x^*P(a)
\]
for $a\in T$ and $x\in X$.
It then follows from Lemma~\ref{lemma:lt_omega} that the
extension of $\lt_\omega$ to a conditional expectation
$E:M_2(\C_0(\QG))\to \alg_{L_\omega(\C_0(\G))}$ is equal to
\[
E  = \begin{bmatrix}  \lt_{|\omega|_r} & \lt_\omega \\
                      \lt_{\conj{\omega}}    & \lt_{|\omega|_l}
    \end{bmatrix}.
\]

\section{Abstract characterisation of the
         subspaces of $\C_0(\QG)$ related to contractive idempotents}

In this section we generalise results of Sections 3 and 4 in \cite{salmi-skalski:idem} (see also \cite{FS} and \cite{S}) describing a natural correspondence between the idempotent states on $\M(\G)$ and certain subalgebras of $\C_0(\QG)$.

Denote the left Haar weight of $\G$ by $\phi$ and the right Haar
weight of $\G$ by $\psi$.
Consider the map $\psitwo:M_2(\C_0(\G))_+ \to [0, +\infty]$
defined by
\[
\psitwo\left(\begin{bmatrix} a_1 & a_2 \\
                             a_3 & a_4 \end{bmatrix}\right)
    = \psi(a_1) + \psi(a_4).
\]
It is easy to check that $\psitwo$ is a densely defined, faithful,
lower semicontinuous weight on $\Mtwo$. Moreover,
for $a_1,a_2,a_3, a_4 \in \C_0(\G)$
\[
\begin{bmatrix} a_1 & a_2 \\
                a_3 & a_4 \end{bmatrix}\in\mathfrak{N}_{\psitwo}
\iff a_1,a_2,a_3, a_4 \in \Npsi,
\]
where $\Npsi=\{x\in \C_0(\QG): \psi(x^*x)< \infty\}$, $\mathfrak{N}_{\psitwo}=\{x\in \Mtwo: \psi^{(2)}(x^*x)< \infty\}$. The weight $\phi^{(2)}$ is defined in an analogous way.
A subset $F$ of $\Mtwo$ will be called
\emph{right invariant} if for all $\nu \in \C_0(\G)^*$ the
matrix-lifted right shift map $R_{\nu}^{(2)}:\Mtwo \to \Mtwo$ leaves
$F$ invariant.

If $\eta$ is a weight on a $C^*$-algebra $\alg$ and $\blg$ is a
$C^*$-subalgebra of $\alg$ such that there exists a conditional
expectation from $\alg$ onto $\blg$ which preserves $\eta$
(i.e.\ $\eta(P(a))= \eta(a)$ where $P$ is the conditional expectation
and $a\in\alg_+$ with $\eta(a)<\infty$), then we call $\blg$ an
\emph{$\eta$-expected subalgebra}.

\begin{theorem} \label{idemp->TRO}
Suppose that $X$ is a nondegenerate TRO in $\C_0(\G)$. If
$\alg_X \subset \Mtwo$ is a right invariant, $\psitwo$-expected subalgebra,
then there exists a contractive idempotent functional $\omega \in
\M(\G)$ such that $X=L_{\omega} (\C_0(\G))$. Moreover
$\alg_X \subset \Mtwo$ is also a $\phitwo$-expected
subalgebra.
\end{theorem}

\begin{proof}
Assume that $X$ satisfies the assumptions of the theorem.
Let $E: \Mtwo \to \alg_X$ be the corresponding $\psitwo$-preserving
conditional expectation. Note first that the right invariance of
$\alg_X$ means that each of the spaces $\la XX^* \ra$, $\la X^*X \ra$
and $X$ are right invariant (as subsets of $\C_0(\G)$). This
in turn allows us to deduce that for any $\fatnu \in \Mtwo^*$ the map
\[
R_{\fatnu}:= \begin{bmatrix} \rt_{\nu_{11}} & \rt_{\nu_{12}} \\
                           \rt_{\nu_{21}} & \rt_{\nu_{22}}
            \end{bmatrix}  %(\fatnu \ot \id_{\Mtwo})\Delta^{(2)}$
\]
leaves $\alg_X$ invariant
(we use the natural identification $\Mtwo^* = M_2(\CG^*)$).
In the next step we show that in fact for all $\fatnu \in \Mtwo^*$
\begin{equation}
E R_\fatnu = R_\fatnu E.  \label{commutexp}
\end{equation}
The method will be similar to that used in Proposition 3.3 of
\cite{salmi-skalski:idem}. Note first that by Lemma \ref{invar} it suffices to show the
above equality for $\fatnu$ in a weak$^*$-dense subset of
$\Mtwo^*$. Assume then that
\[
\fatnu = \begin{bmatrix} \nu_{11} &
  \nu_{12} \\ \nu_{21} & \nu_{22} \end{bmatrix} \in M_2 (\Lonenat)
\]
and let $\rho_{ij}\in \Lonenat$ be such that
$\rho_{ij} = \nu_{ij}^{\sharp}$ for $i,j=1,2$. Further write $\fatrho =
\begin{bmatrix} \rho_{11} & \rho_{12} \\ \rho_{21} &
  \rho_{22} \end{bmatrix}$ and recall that it follows from the
computations in the proof of Proposition 3.3 in
\cite{salmi-skalski:idem} that for any $a,b \in \Npsi$ the following
equality holds:
$ \psi (a^* R_{\nu_{ij}}(c)) = \psi (R_{\rho_{ij}} (a)^* c)$.

As $R_\fatnu (\alg_X) \subset \alg_X$, by Lemma 1.3 of \cite{salmi-skalski:idem} to prove equality
\eqref{commutexp} it suffices to show that  $R_\fatnu
((\alg_X)^{\perp}) \subset (\alg_X)^{\perp}$,
where
\[ (\alg_X)^{\perp} = \set{\mathbf{a}\in\Mtwo\cap\mathfrak{N}_{\psitwo}}{\psitwo(\mathbf{c}^*\mathbf{a})= 0 \textup{ for all } \mathbf{c} \in\alg_X\cap\mathfrak{N}_{\psitwo}}. \]
Let then $\mathbf{a}
= \begin{bmatrix} a_{11} & a_{12} \\ a_{21} & a_{22} \end{bmatrix} \in
\alg_X \cap \mathfrak{N}_{\psitwo}$ and $\mathbf{c} = \begin{bmatrix}
  c_{11} & c_{12} \\ c_{21} & c_{22} \end{bmatrix} \in
(\alg_X)^\perp$.
We thus have, by the formulas above,
\begin{align*}
\psitwo( \mathbf{a}^* R_\fatnu (\mathbf{c}))
 &= \psi (a_{11}^* R_{\nu_{11}} (c_{11}) + a_{21}^* R_{\nu_{21}}(c_{21}))
  + \psi (a_{12}^* R_{\nu_{12}} (c_{12}) + a_{22}^* R_{\nu_{22}} (c_{22})) \\
 &= \psi(R_{\rho_{11}}(a_{11})^* c_{11} +  R_{\rho_{21}}(a_{21})^* c_{21}
  + R_{\rho_{12}}(a_{12})^* c_{12} + R_{\rho_{22}}(a_{22})^* c_{22})\\
 &= \psitwo( R_\fatrho (\mathbf{a})^*  \mathbf{c}) =0,
\end{align*}
where the last formula follows from the definition of
$(\alg_X)^{\perp}$. This finishes the proof of equality
\eqref{commutexp}.

The first consequence of \eqref{commutexp} is the following.
Consider $\epsilon_{11}:= \begin{bmatrix} \epsilon &  0 \\0 &
  0 \end{bmatrix} \in \Mtwo^*$, where $\epsilon\in \M(\G)$ denotes the counit. It is easy to see that
$R_{\epsilon_{11}}=  \begin{bmatrix} \id_{\CG} &  0 \\0 &
  0 \end{bmatrix}$, so the analysis of the commutation relations of
$E$ with $R_{\epsilon_{11}}$ (and its analogues) implies that $E$
is a so-called \emph{Schur map}, i.e.\
\[
E = \begin{bmatrix} E_{11} & E_{12} \\
                    E_{21} & E_{22} \end{bmatrix},
\]
where $E_{ij}:\CG \to \CG$ are completely
contractive projections; in particular $E_{12}$ is a completely
contractive projection onto $X$. It is further easy to see that each
of the maps $E_{ij}$ commutes with all the maps $R_{\nu}$, $\nu \in
\CG^*$. Hence by Lemma \ref{invar} there exist
contractive functionals $\omega_{ij} \in \CG^*$ such that $E_{ij} =
L_{\omega_{ij}}$ for $i,j=1,2$. Each $\omega_{ij}$ is a contractive
idempotent functional. It suffices to put $\omega:=\omega_{12}$ to end
the first part of the proof.

The second part follows because
in Section \ref{TROSection} we showed that if $X=L_{\omega} (\CG)$, then the map
\[
E':= \begin{bmatrix}
         L_{\rom} & L_{\omega} \\
         L_{\overline{\omega}} & L_{\lom}
\end{bmatrix}
\]
is a conditional expectation onto $\alg_X$;
it is easy to check that it preserves $\phitwo$.
Moreover, $E' = E$ by the uniqueness of the extension of
the TRO expectation $L_\omega$ (Theorem 2.1 of \cite{TRO-paper}).
\end{proof}

Recall that $\QG$ is said to be \emph{unimodular} if $\phi=\psi$.
If $\G$ is unimodular and $\alg$ is a $\psi$-expected $C^*$-subalgebra
of $\C_0(\G)$, then we simply call it \emph{expected}.

\begin{corollary}
Assume that $\G$ is unimodular.
There is a bijective correspondence between
\begin{rlist}
\item contractive idempotent functionals $\omega \in \M(\G)$;
\item nondegenerate TROs $X$ in $\C_0(\G)$ such that their associated
  $C^*$-algebra $\alg_X \subset \Mtwo$ is right invariant and
  expected.
\end{rlist}
In that case the conditional expectation $E:\Mtwo \to\alg_X$
preserving the invariant weight is given by the formula
\[
E= \begin{bmatrix} L_{\rom} & L_{\omega} \\
                   L_{\overline{\omega}} & L_{\lom} \end{bmatrix}.
\]
\end{corollary}

\begin{proof}
An easy consequence of the last theorem and the results of Section 2.
\end{proof}

We say that a TRO $X$ contained in a $C^*$-algebra $\alg$ is
\emph{balanced} if $\la X^*X \ra = \la X X^* \ra$.
A right invariant C*-subalgebra $\clg$ of $\C_0(\G)$
(such as $\la X^*X \ra$ or $\la X X^* \ra$) is said to be
\emph{symmetric} if
$V^*(1\ot c)V$ is in $M\bigl(\K_{\ltwo(\G)}\ot \clg\bigr)$
for every $c\in \clg$.
The symmetry condition originates from \cite{tomatsu:coideal}
and \cite{S}.

\begin{corollary}
There is a bijective correspondence between
\begin{rlist}
\item contractive idempotent functionals $\omega\in \M(\G)$ whose
  left \textup{(}equivalently, right; equivalently left and right\textup{)} 
  absolute values are Haar idempotents;
\item nondegenerate balanced TROs in $\C_0(\G)$ such that their
  associated $C^*$-algebra $\alg_X \subset\Mtwo$ is right-invariant
  and $\psitwo$-expected \textup{(}equivalently, expected\textup{)}
  and the algebra $\la  XX^* \ra = \la X^* X \ra$ is symmetric.
\end{rlist}
\end{corollary}

(In fact, if either $\la XX^* \ra$ or  $\la X^* X \ra$ is symmetric,
then they are equal.)

\begin{proof}
Let $\omega\in \M(\G)$ be a contractive idempotent and put
$X=L_{\omega} (\CG)$. Theorem 3.7 of \cite{salmi-skalski:idem} implies
that $|\omega|_l$ is a Haar idempotent if and only if
$L_{|\omega|_l} (\CG) = \la X^* X\ra$ is symmetric. Moreover in that case
$|\omega|_l = |\omega|_r$ by Theorem~\ref{thm:cts-fact}
(so also $\la X^*X \ra = \la X X^* \ra$), and the
conditional expectation
\[
E= \begin{bmatrix} L_{\rom} & L_{\omega} \\
                   L_{\overline{\omega}} & L_{\rom} \end{bmatrix}
\]
onto $\alg_X$ preserves both $\psitwo$ and $\phitwo$ in view of
Proposition 3.12 of \cite{salmi-skalski:idem}.

Conversely, if $X$ is a nondegenerate TRO and $\alg_X$ is
$\psitwo$-expected, then Theorem \ref{idemp->TRO} yields a contractive
idempotent $\omega \in \CG^*$ such that $X=L_{\omega}(\CG)$. Moreover
the fact that $\la X^*X \ra = L_{|\omega|_l}$ is symmetric implies
that $|\omega|_l$ is a Haar idempotent, so $|\omega|_l = |\omega|_r$
and $L_{|\omega|_l}$ preserves both invariant weights. Hence $\alg_X$
is expected.
\end{proof}

\section{Examples of contractive idempotents and some special cases}

In this section we collect certain examples and special cases of the
results obtained in the previous sections.

\begin{propn}[\cite{greenleaf:homo}] \label{groupcontidemp}
Let $G$ be a locally compact group. Then every
contractive idempotent measure $\mu$ on $G$
is of the form
\[
d\mu(s) = \chi(s)\,dm_H(s)
\]
where $m_H$ is the Haar measure of a compact subgroup
$H$ of $G$ and $\chi$ is a continuous character of $H$.
\end{propn}

\begin{proof}
Every idempotent state on $G$ is a Haar idempotent, so
Theorem \ref{thm:cts-fact} applies. Moreover, $|\mu|_r = |\mu| = m_H$
for some compact subgroup $H$ of $G$.
By Theorem~\ref{thm:cts-fact}, $u$ is a continuous character
in $\C(H)$, so we are done.
\end{proof}

In view of  the above proposition it is natural to ask what do
contractive idempotent functionals look like in the dual case, i.e.\ when
$\C_0(\G)=\C^*(\Gamma)$ for a locally compact amenable group $\Gamma$.
The answer was given by
Theorem 2.1 (i) of \cite{IS1} (in a more general context where
$\Gamma$ need not be amenable); here we show how the special case of
that result can be quickly deduced via the techniques developed in
this paper.

\begin{propn}[\cite{IS1}] \label{dual}
Let $G$ be an amenable locally compact group and let $\omega \in
\C^*(G)^*= B(G)$. Then $\omega$ is a contractive idempotent if and
only if it is of the form $1_C$, where $C$ is an open coset of
$G$. Moreover $|\omega|_r$ \textup{(}equivalently, $|\omega_l|$\textup{)}
is a Haar idempotent if and only if $C^{-1}C$ is a normal subgroup of $G$.
\end{propn}

\begin{proof}(Sketch)
Let $\omega\in B(G)$. It is easy to check that the map
$L_{\omega}:M(\C^*(G)) \to M(\C^*(G))$ is then the Schur multiplier
corresponding to the function $\omega$. Thus the fact that
$L_{\omega}$ is an idempotent implies that $\omega$ is a
characteristic function. Due to continuity it must be a characteristic
function of a clopen set, say $C$. Further if $\omega$ is a
contractive idempotent, (strict extensions of) the equalities
displayed in the proof of Theorem \ref{thm:TRO expectation} applied to
generating unitaries in $M(\C^*(G))$ imply that if $g_1, g_2, g_3 \in
C$, then also $g_1 g_2^{-1} g_3 \in C$. By Proposition 1.1 of
\cite{IS1} $C$ must then be a coset of $G$.
The other implication is easy to check, for example using the mentioned before fact
that TRO conditional expectations are necessarily contractive.

Finally the last statement in the proposition is a consequence of the
results in Section 7 of \cite{S}; it can be also checked directly.
\end{proof}

In fact our results allow us also to produce in some cases a complete list of contractive idempotents on a given (locally) compact quantum group $\QG$. The next proposition presents a result of that type for $SU_q(2)$ (see \cite{woronowicz87b} for the definition of the latter quantum group).

\begin{propn}
Let $q\in (0,1)$. The contractive idempotents on $\SU_q(2)$ are
the Haar state and the Haar idempotents associated with
the subgroups $\T$ and $\Z_n$ combined with the characters of
those subgroups.
\end{propn}

\begin{proof}
By Theorem~5.1 of \cite{FST},
all idempotent states of $\SU_q(2)$ are Haar idempotents
associated with the subgroups listed in the statement.
Then the statement follows from Theorem~\ref{thm:cts-fact}
once we note that  the algebra $\C(\SU_q(2))$ itself does not have any
nontrivial group-like unitaries (i.e.\ $SU_q(2)$ does not have any nontrivial one-dimensional
representations).
\end{proof}

In the case presented above there exist relatively few contractive idempotents and all of them fit in the scheme described in Theorem \ref{thm:cts-fact}. Below we discuss briefly an example of a genuinely quantum (i.e.\ neither commutative nor cocommutative) group where there exist non-Haar idempotent states and thus also more complicated contractive idempotent functionals.

\begin{example}
Consider the quantum group $\SU_{-1}(2)$, which is of Kac type. It is well-known that $\SU_{-1}(2)$ admits the dual of $D_{\infty}$ (the infinite dihedral group) as a compact quantum subgroup, see for example \cite{podles95} or \cite{FSTKaccase}. In the latter notes a relevant surjective $^*$-homomorphism $\pi:\C(\textup{SU}_{-1}(2))\to \C^*(D_{\infty})$ exchanging the comultiplications is shown to be determined by the conditions
\[ \pi(u_{11})= \pi(u_{22}) = \frac{1}{2} t (1+z), \pi(u_{12})= \pi(u_{21}) = \frac{1}{2} t (1-z),\]
where $U=(u_{ij})_{i,j=1}^2 \in M_n (\C(\SU_{-1}(2)))$ is the defining fundamental representation of $\SU_{-1}(2)$ and $t,z$ are the canonical unitary generators of $\C^*(D_{\infty})$ satisfying the commutation relations $t^2=1$ and $ztz=t$. Thus choosing a contractive idempotent functional $\mu$ on  $\C^*(D_{\infty})$ whose left or right absolute value is not a Haar idempotent (i.e., by Proposition \ref{dual}, choosing a coset of a non-normal subgroup of $D_{\infty}$) and composing it with $\pi$ yields a contractive idempotent functional on $\C(\SU_{-1}(2))$ whose absolute value is not a Haar idempotent. A concrete example can be obtained by putting $C=\{t,tz\}$ -- note we  identify the unitary generators of $\C^*(D_{\infty})$ with elements of $D_{\infty}$. In fact the results of \cite{FSTKaccase} and \cite{tomatsu08} make it possible in principle to provide  a full list of the contractive idempotent functionals on $\C(\SU_{-1}(2))$.
\end{example}

\begin{acknow}
The work on this article was completed during the Research in Pairs
visit to the Mathematisches Forschungsinstitut Oberwolfach in August
2012. We are very grateful to MFO for providing ideal research
conditions.

PS thanks the Emil Aaltonen Foundation for support.
We also thank the Finnish Academy of Science and Letters, Vilho, Yrj\"o
and Kalle V\"ais\"al\"a Foundation, for supporting a visit of AS to
the University of Oulu. AS was partly supported by the National Science Centre (NCN) grant no.~2011/01/B/ST1/05011. NS was partially supported by the NSERC grant DG 312515-2010.

We thank the referee for helpful comments.
\end{acknow}

\end{document}